\documentclass[11pt, reqno]{amsart}
\usepackage{graphicx}
\pdfoutput=1
\usepackage{amsfonts, amsmath, amssymb}
\usepackage{hyperref}
\usepackage{paralist}

\usepackage{thmtools}
\declaretheoremstyle[bodyfont=\normalfont]{noncursive}
\declaretheorem{theorem}
\declaretheorem[numberwithin=section]{lemma}
\declaretheorem[numberlike=lemma]{conjecture}
\declaretheorem[numberlike=lemma]{proposition}

\declaretheorem[style=noncursive, numberlike=lemma]{definition}
\declaretheorem[style=noncursive, numberlike=lemma]{example}
\declaretheorem[style=noncursive, numberlike=lemma]{remark}

\declaretheorem[style=noncursive, numberlike=lemma]{claim}
\newcommand{\im}{\ensuremath{\mbox{\rm Im}\, }}

\newcommand{\CC}[1]{\mathbb{C}^{#1}}
\newcommand{\C}{\mathbb{C}}

\numberwithin{equation}{section}

\def\<{\langle}
\def\>{\rangle}
\numberwithin{equation}{section}
\def\<{\langle}
\def\>{\rangle}

\def\-{\overline}

\def\-{\overline}

\def\-{\overline}

\begin{document}

\title[Regularity of CR mappings of codimension one]{ Regularity of CR-mappings of codimension one
\\ into Levi-degenerate hypersurfaces }

\author{Ilya Kossovskiy}
\address{Department of Mathematics, University of Vienna // Department of Mathematics, Masaryk University in Brno}
\email{ilya.kossovskiy@univie.ac.at, kossovskiyi@math.muni.cz}

\author{Bernhard Lamel}
\address{Department of Mathematics,  University of Vienna,  Vienna,  Austria}
\email{bernhard.lamel@univie.ac.at}

\author{Ming Xiao}
\address{Department of
Mathematics,  University of California San Diego, La Jolla, USA}
\email{m3xiao@ucsd.edu}

\begin{abstract}
We provide regularity results for CR-maps from a real hypersurface in $n$-dimensional complex space to a {\em Levi-degenerate target} hypersurface in a $n+1$-dimensional space.  We address both the real-analytic and the  smooth case. Our results allow immediate applications to the study of proper holomorphic maps between Bounded Symmetric Domains.

\end{abstract}

\maketitle

\section{Introduction}

This paper is devoted to establishing
smooth and real-analytic versions of the {\em Schwarz reflection principle} for holomorphic
maps in several complex variables. In the real-analytic version of the reflection principle,  we investigate conditions under which a {\em CR-map} between real submanifolds in complex space (or a holomorphic map between {\em wedges} attached to  real submanifolds) extends holomorphically to an open neighborhood of the source manifold. In the smooth version,  we ask for conditions under which a CR-map between real submanifolds in complex space has  higher regularity than the given one. Problems of this type have
attracted considerable attention since the work of Fefferman \cite{Fe},  Lewy \cite{Le},  and Pinchuk \cite{Pi}.
In the {\em equidimensional} case,  the reflection principle is understood quite well due to the extensive research in this direction.  We refer the reader to e.g. \cite{BER, Fr1, KL,analytic} for detailed surveys and references
related to this research,  as well as for the most up-to-date results.

In this paper,  we study aspects of the
 regularity problem  for CR-mappings  between CR-manifolds $M$ and $M'$ of {\em different} dimension. This has been an extensively developing direction since the pioneering work of  Webster \cite{W}, Faran \cite{faran}, and Forstneri\v{c} \cite{Fr1}. We shall note that the case of different dimensions is far more difficult than the equi-dimensional one, and much less is known in this setting. For an overview of existing results in the real-analytic case, we refer to the recent work of Berhanu and the first author \cite{BX1}. 
 
 The regularity problem in the {\em smooth
category} rather than in the real-analytic one (in what follows,  by ``smooth'' we refer to the $C^\infty$ smoothness,  if not otherwise stated) seems to be even more difficult due to lack of techniques. Starting from the work  of Forstneri\v{c} \cite{Fr1} and Huang \cite{Hu1},  \cite{Hu2},  the expected type of regularity of a finitely smooth CR-map between smooth CR-manifolds is {\em its $C^\infty$ smoothness at a generic point.} One of the main tools for obtaining  results in this line was introduced in the work \cite{L1, L2, L3} by the second author,  which is the notion of {\em $k$-nondegeneracy} of a CR-mapping. The latter is used for studying differential systems associated with CR-mappings.  In particular,  this tool was applied by Berhanu and the third author for studying the situation when the target manifold is {\em Levi-nondegenerate}. In the work \cite{BX1},  a smooth version of the reflection principle is established for CR-mappings from an abstract CR-manifold to a strongly pseudoconvex hypersurface.
In particular,  it solves a conjecture formulated earlier  by  Huang \cite{Hu2} and also 
reproves a conjecture of  Forstneri\v{c} \cite{Fr1}  consequently.
In \cite{BX2},  this type of result is extended for CR-mapping into Levi-nondegenerate CR-submanifolds of  hypersurface type with certain conditions on the signature.  These results in particular show  that if $F: M \rightarrow M'$ is a CR-transversal CR-mapping of class $C^2$ from a real-analytic (resp. smooth) strictly pseudoconvex hypersurface $M \subset \mathbb{C}^n$ into a real-analytic (resp. smooth) Levi-nondegenerate hypersurface $M' \subset \mathbb{C}^{n+1}$,   then $F$ is real-analytic (resp. smooth) on a dense open subset of $M$ (we mention that when $F$ is assumed to be $C^{\infty}$,  the result in the real-analytic case was proved in \cite{EL}). 

However,  the case when the target is {\em Levi-degenerate} remains widely open,  in both smooth and real-analytic categories, and very little is known in this setting. In the real-analytic case, a number of very interesting results in the latter direction were obtained  in the recent paper of Mir \cite{M1}.

\medskip

The main goal of this  paper is to extend the reflection principle for CR-maps of real hypersurfaces in complex space to the setting when the target hypersurface $M' \subset\CC{n+1}$ is {\em Levi-degenerate},  while the source $M\subset\CC{n}$ is strictly pseudoconvex.

First,  we  obtain in the paper the generic analyticity property (resp. the generic smoothness property)  for finitely smooth CR-maps between real-analytic (resp. smooth) real hypersurfaces of different dimensions  with minimal assumptions for the target. Namely,  in the real-analytic case,  we assume the target $M'$ to be merely {\em holomorphically nondegenerate}. Clearly,  for any given source,  the latter assumption can not be relaxed further (see Example 1.1 below). In the smooth case,  we assume the {\em finite nondegeneracy} of the target. For definitions of different notions of nondegeneracy,  see Section 2.

Second,  we establish in the paper the {\em everywhere analyticity} (resp. {\em everywhere smoothness}) of CR-maps  in the case when the target belongs to the class of {\em uniformly $2$-nondegenerate hypersurfaces}. The latter class of hypersurfaces is of fundamental importance in Complex Analysis and Geometry. Uniformly $2$-nondegenerate hypersurfaces   have been recently studied intensively (e.g. Ebenfelt \cite{E1, E2},  Kaup and Zaitsev \cite{KaZa},  Fels and Kaup \cite{FK1, FK2},  Isaev and Zaitsev \cite{IZ},  Medori and Spiro \cite{MS},  Kim and Zaitsev \cite{KiZa},  Beloshapka and the first author \cite{BK}). These hypersurfaces naturally occur as boundaries of Bounded Symmetric Domains (see,  e.g.,  \cite{KaZa},  \cite{XY}  for details), and in this way CR-maps into uniformly $2$-nondegenerate hypersurfaces become important for understanding proper holomorphic maps between the respective Bounded Symmetric Domains (on the latter subject, see e.g. the work of Mok \cite{mok1,mok2}  and references therein).  Uniformly $2$-nondegenerate hypersurfaces  occur as well  as homogeneous holomorphically nondegenerate CR-manifolds \cite{FK1,FK2}. We  also note that the study of CR-embeddings of strictly-pseudoconvex hypersurfaces into $2$-nondegenerate  hypersurfaces performed in the present paper is important for understanding the geometry of the latter class of CR-manifolds (see,  e.g.,  \cite{BK}).

\medskip

We shall now formulate our main results. Let us recall that a map $F \colon M \to M'$ 
between real hypersurfaces is said to be {\em CR-transversal} at $p\in M$ 
if $T^{(1,0)}_{F(p)} M'  +  T^{(0,1)}_{F(p)} M' + dF (\C T_p M) = \C T_{F(p)} M'$. The
definition of uniform $2$-nondegeneracy is going to be given below in \autoref{def:basic}.

\begin{theorem}\label{T0}
Let $M \subset \mathbb{C}^n\,  (n \geq 2)$ be a strongly pseudoconvex real-analytic (resp. smooth)  hypersurface,  and $M' \subset \mathbb{C}^{n+1}$ a uniformly
$2-$nondegenerate real-analytic (resp. smooth)  hypersurface. Assume that $F=(F_{1}, ..., F_{n+1}): M \mapsto M'$ is a CR-transversal CR-mapping of class $C^2.$ Then $F$ is real-analytic (resp. smooth) everywhere on $M.$
\end{theorem}

We note that \autoref{T0} has direct applications to the study of rigidity of  proper holomorphic maps between bounded symmetric domains (see the work \cite{XY} of Yuan and the third author, where certain rigidity results for holomorphic proper maps from the complex unit ball 
to the Type IV bounded symmetric domain $D^{IV}_m$ are obtained).   We also note that \autoref{T0} somehow parallels a theorem proved by Mir \cite{M1} and establishing the  analyticity of CR-maps (at a generic point) in the situation when the source $M$ is real-analytic and minimal while the target is the well known uniformly $2$-nondegenerate hypersurface called {\em the tube over the future light cone}:

\begin{equation}\label{cone}
\mathbb T_{n+1}=\left\{(z_1, ...., z_{n+1})\in\CC{n+1}:\, \, (\im z_1)^2+\cdots+(\im z_n)^2=(\im z_{n+1})^2\right\}.
\end{equation}

Next,  in the more general setting of $M'$,  we prove

\begin{theorem}\label{T1}
Let $M \subset \mathbb{C}^n\,  (n \geq 2)$ be a strongly pseudoconvex smooth             hypersurface,  and
$M' \subset \mathbb{C}^{n+1}$ an everywhere finitely nondegenerate  smooth hypersurface.  Let  $F=(F_{1}, ..., F_{n+1}): M \mapsto M'$ be a CR-transversal CR-mapping of class $C^2.$ Then $F$ is smooth on a dense open subset of $M.$
\end{theorem}

If the target hypersurface is real-analytic, a stronger assertion holds. 
\begin{theorem}\label{T1a}
Let $M \subset \mathbb{C}^n\,  (n \geq 2)$ be a strongly pseudoconvex smooth             hypersurface,  and
$M' \subset \mathbb{C}^{n+1}$ a holomorphically 
nondegenerate real-analytic hypersurface.  
Let  $F=(F_{1}, ..., F_{n+1}): M \mapsto M'$ 
be a CR-transversal CR-mapping of class $C^2.$ Then $F$ is smooth on a dense open subset of $M.$
\end{theorem}

A similar result holds in the real-analytic category:
here we prove the stronger result, which 
does (only) require $M'$ to be holomorphically nondegenerate.

\begin{theorem} \label{T2}
Let $M \subset \mathbb{C}^n\,  (n \geq 2)$ be a strongly pseudoconvex  real-analytic hypersurface,  and
$M' \subset \mathbb{C}^{n+1}$ a holomorphically nondegenerate real-analytic  hypersurface. Assume that $F=(F_{1}, ..., F_{n+1}): M \mapsto M'$ is a CR-transversal CR-mapping of class $C^2.$ Then $F$ is real-analytic on
a dense open subset of $M.$
\end{theorem}

\bigskip

As was  mentioned above,  for any given $M$,  one cannot drop the holomorphic nondegeneracy assumption when expecting the generic analyticity of CR-embeddings $F:\, M\mapsto M', \, M'\in\CC{n+1}$ (see Example 1.1 below).
The transversality assumption on $F$ cannot be dropped either. See [BX2] for an example
where $F$ (being not transversal) is not smooth on any open subset of $M$. Thus,  the assertion of \autoref{T2} is in a sense optimal.

\begin{example}\label{E1}
Let $M \subset \mathbb{C}^n, \,  n \geq 2$ be a strongly pseudoconvex  hypersurface.
Consider the holomorphically degenerate hypersurface $M' =M\times\CC{}\subset\CC{n+1}$.
Let  $f$
be a $C^{2}$ CR function on $M$ which is not smooth on any open subset of $M$.
Then $F(Z):=(Z,  f(Z)), \, Z\in M$ is a CR-transversal map of class $C^2$ from $M$ to $M'.$ Clearly,  $F$ is not smooth on any open subset of $M.$
\end{example}

The following example shows also that one cannot expect
$F$ to be real-analytic everywhere on $M$ in the setting of Theorems 3.

\begin{example}\label{E2}
Let $M \subset \mathbb{C}^2$ be the strongly pseudoconvex real hypersurface defined by
$$|z|^2+|w|^2+|1-w|^{10}=1$$
near $(0, 1)$,   where $(z, w)$ are the coordinates in $\mathbb{C}^2.$ Let $M' \subset \mathbb{C}^3$ be the holomorphically nondegenerate real hypersurface  defined by
$$|z_1|^2+|z_2|^2+|z_3|^4=1,$$ 
where $(z_1, z_2,  z_3)$ are the coordinates in $\mathbb{C}^3.$ Consider the map
$$F=(z,  w,  (1-w)^{5/2})$$
from one side of $M: \{ |z|^2+|w|^2+|1-w|^{10} < 1 \}$ to $\mathbb{C}^3.$ It is easy to see $F$ extends $C^2-$smoothly up to
$M$,   sending $M$ to $M'.$ However,  $F$ is not even $C^3$ at the point $(0, 1).$
\end{example}

\bigskip

We,  however,  hope that the following is true.
\begin{conjecture}
For any integer $n\geq 2$, there exists an integer $k = k(n)$ such that the following holds. Let $M \subset \mathbb{C}^n\,  (n \geq 2), \, M' \subset \mathbb{C}^{n+1}$ 
be  real-analytic (resp. smooth) hypersurfaces that are finitely nondegenerate (on some dense open subsets), and $F=(F_{1}, ..., F_{n+1}): M \to M'$ be a 
CR-transversal CR-mapping of class $C^{k}$. Then $F$ is 
real-analytic (resp. smooth)  on
a dense open subset of $M.$
\end{conjecture}
(In the real-analytic version of the Conjecture, we may replace the condition on $M$ by its holomorphic nondegeneracy). 

\medskip

The paper is organized as follows. In Section 2,  we present some preliminaries on the degeneracy of CR-submanifolds and CR-mappings. Section 3 is devoted to a normalization result for a CR-map between hypersurfaces satisfying the assumptions of Theorems 1-3. It will be applied in later arguments. Theorems \ref{T0}-\ref{T2} will be proved in Sections 4-6.

\bigskip

{\bf Acknowledgement:} The third author thanks  Shiferaw Berhanu and Xiaojun Huang for helpful comments. All of the authors thank two anonymous 
referees for several helpful comments and suggestions. 

The first author was supported by the FWF (Austrian Science Fund) Projects P29468 and I3472 and the GACR (Czech Grant Agency) grant 17-19437S. The second author is supported by the FWF (Austrian Science Fund) Project I3472. The third author is supported in part by National Science Foundation grant DMS-1800549.

\section{Preliminaries}
In this section,  we recall various notions of degeneracy in CR geometry,  and their relations.  The following definition is introduced in \cite{BHR}.
\begin{definition}\label{def:basic}
Let $M$ be a smooth generic submanifold in $\mathbb{C}^N$ of CR-dimension $n$ and  real codimension $d$,  and $p \in M$. Let $\rho=(\rho_{1}, ..., \rho_{d})$ be the defining function of $M$ near $p$,  
and choose a basis $L_{1}, ..., L_{n}$ of CR vector fields near $p.$ For a multiindex
$\alpha=(\alpha_{1}, ..., \alpha_{n})$,   write $L^{\alpha}=L_{1}^{\alpha_1}...L_{n}^{\alpha_n}.$
Define the increasing sequence of subspaces $E_{l}(p)$ ($0 \leq l$) of $\mathbb{C}^{N}$ by

$$E_{l}(p)=\mathrm{Span}_{\mathbb{C}}\{L^{\alpha} \rho_{\mu, Z}(Z, \overline{Z})|_{Z=p}: 0 \leq |\alpha| \leq l,  1 \leq \mu \leq d\}.$$
Here $\rho_{\mu, Z}=(\frac{\partial \rho_{\mu}}{\partial z_{1}}, \cdots, \frac{\partial \rho_{\mu}}{\partial z_{N}})$,   and $Z=(z_{1}, \cdots, z_{N})$ are the coordinates in $\mathbb{C}^{N}.$ We say that $M$ is {\em $k-$nondegenerate at $p,  \, \,  k \geq 1$} if
$$E_{k-1}(p) \neq E_{k}(p) = \mathbb{C}^{N}.$$ 
We say $M$ is $k-$degenerate at $p$ if $E_{k}(p) \neq \mathbb{C}^N$, and 
we say that $M$ is uniformly $k$-nondegenerate if it is $k$-nondegenerate at 
every point $p\in M$. 
\end{definition}

We say $M$ is (everywhere) finitely nondegenerate if $M$ is $k(p)-$nondegenerate at every $p \in M$ for some integer $k(p)$ depending on $p.$ A smooth CR-manifold $M$ of hypersurface type is Levi-nondegenerate at $p \in M$ if and only if $M$ is $1-$nondegenerate at $p.$ This notion of degeneracy is then generalized to
CR-mappings by the second author [La1] as follows.

\begin{definition}
Let ${M} \subset \mathbb{C}^N,  {M}' \subset \mathbb{C}^{N'}$ be two generic CR-submanifolds of CR dimension $n,  ~n'$,  respectively.  Let $H: {M} \rightarrow {M}'$ be a CR-mapping of class $C^r$ near $p_{0} \in {M}.$ Let $\rho=(\rho_{1}, \cdots, \rho_{d'})$ be local defining functions for ${M}'$ near $H(p_{0})$,   and choose a basis $L_{1}, \cdots, L_{n}$ of CR vector fields for ${M}$ near $p_{0}.$ If $\alpha=(\alpha_{1}, \cdots, \alpha_{n})$ is a multiindex,  write $L^{\alpha}=L_{1}^{\alpha_{1}} \cdots L_{n}^{\alpha_{n}}.$ Define the increasing sequence of subspaces $E_{l}(p_{0}) (0 \leq l \leq r)$ of $\mathbb{C}^{N'}$ by

$$E_{l}(p_{0})=\mathrm{Span}_{\mathbb{C}}\{L^{\alpha} \rho_{\mu, Z'}(H(Z), \overline{H(Z)})|_{Z=p_{0}}: 0 \leq |\alpha| \leq l,  1 \leq \mu \leq d'\}.$$
Here $\rho_{\mu, Z'}=(\frac{\partial \rho_{\mu}}{\partial z'_{1}}, \cdots, \frac{\partial \rho_{\mu}}{\partial z'_{N'}})$,   and $Z'=(z'_{1}, \cdots, z'_{N'})$ are the coordinates in $\mathbb{C}^{N'}.$ We say that {\em $H$ is $k_{0}-$nondegenerate at $p_{0}\, \, (0 \leq k_{0} \leq r)$} if
$$E_{k_{0}-1}(p_{0}) \neq E_{k_{0}}(p_{0}) = \mathbb{C}^{N'}.$$

\end{definition}

A manifold $M$ is $k_0-$nondegenerate if and only if the identity map from $M$ to $M$ is $k_0-$nondegenerate. For a real-analytic submanifold,  we also introduce the notion of holomorphic degeneracy.

\begin{definition}
A real-analytic submanifold $M \subset \mathbb{C}^N$ is {\em holomorphically nondegenerate at} $p \in M$ if there is no germ at $p$ of a holomorphic vector field $X$ tangent to $M$ such that $X|_{M} \not \equiv 0.$ We shall also say that {\em $M$ is holomorphically nondegenerate} if it is so at every point of it.
\end{definition}

We recall the following proposition about $k-$nondegeneracy and holomorphic nondegeneracy.
For a proof of this,  see [BER].
\begin{proposition}
Let $M \subset \mathbb{C}^N$ be a connected real-analytic generic manifold with CR dimension $n.$ Then the following conditions are equivalent:
\begin{itemize}
\item $M$ is holomorphically nondegenerate.
\item $M$ is holomorphically nondegenerate at some point $p \in M.$
\item $M$ is $k-$nondegenerate at some point $p \in M$ for some $k \geq 1.$
\item There exists $V$,   a proper real-analytic subset of $M$ and an integer $l=l(M),  1 \leq l(M) \leq n$,   such that $M$ is $l-$nondegenerate at every $p \in M \setminus V.$
\end{itemize}
\end{proposition}

\section{Normalization}
In the section,  we prove an auxiliary normalization result for CR-maps (\autoref{rknormalization} below), following the lines introduced by Huang \cite{Hu3},  in the following setting. Let $M \subset \mathbb{C}^n\,  (n \geq 2)$ be a  strongly pseudoconvex real-analytic (resp. smooth)  hypersurface defined near a point $p_0\in M$,  and $M' \subset \mathbb{C}^{n+1}$ a real-analytic (resp. smooth)  hypersurface which is Levi-degenerate at a point $q_0\in M'$. Assume that $F=(F_{1}, ..., F_{n+1}): M \mapsto M'$ is a CR-transversal CR-mapping of class $C^2$ near $p_{0}$ with $F(p_{0})=q_{0}.$
 We assume,  after a holomorphic change of coordinates in $\mathbb{C}^n$,   $p_{0}=0$ and that $M$ is defined near $0$ by
\begin{equation}
r(Z, \overline{Z})=-\mathrm{Im}z_{n}+ \sum_{i=1}^{n-1}|z_i|^2+\psi(Z, \overline{Z}), 
\end{equation}
where $Z=(z_1, ..., z_n)$ are the coordinates in $\mathbb{C}^n$,   $\psi(Z, \overline{Z})=O(|Z|^3)$ is real-analytic (resp. smooth) function defined near $0.$

After a holomorphic change of coordinates in $\mathbb{C}^{n+1}$,   we assume that $q_{0}=F(p_{0})=0$ and
that $M'$ is locally defined  near $0$ by
\begin{equation}
\rho(W, \overline{W})=-\mathrm{Im}w_{n+1}+\tilde{W}U\overline{\tilde{W}}^t + \phi(W, \overline{W}), 
\end{equation}
for some Hermitian $n \times n$ matrix $U$. Here
$W=(\tilde{W},  w_{n+1})=(w_1, ..., w_n,  w_{n+1})$ are the coordinates in $\mathbb{C}^{n+1}$,  
$\phi(W, \overline{W})=O(|W|^3)$ is a real-analytic (resp. smooth) function defined near $0.$

If we write $F=(\tilde{F}, F_{n+1})=(F_{1}, ..., F_{n}, F_{n+1})$,   then
$F$ satisfies:
\begin{equation}\label{eqn}
-\frac{F_{n+1}-\overline{F_{n+1}}}{2i}+\tilde{F}U\overline{\tilde{F}}^t+ \phi(F, \overline{F})=0, 
\end{equation}
along $M.$ Since $F$ is CR-transversal,  we get $\lambda:=\frac{\partial F_{n+1}}{\partial s}|_0 \neq 0$,   where we write $z_n=s+it$(cf. \cite{BER}).  Moreover,  (\ref{eqn}) shows that the imaginary part of $F_{n+1}$ vanishes to second order at the origin,  and so the number $\lambda$ is real. By applying the change of coordinates in $\mathbb{C}^{n+1}$: $\tau(w_1, ..., w_n,  w_{n+1})=(w_1, ..., w_n,  -w_{n+1})$ if necessary,  we may assume that $\lambda >0.$ Let us write
\begin{equation}\label{eqnlj}
L_{j}=2i\left(\frac{\partial r}{\partial \overline{z_{n}}}\frac{\partial}{\partial \overline{z_{j}}}
+\frac{\partial r}{\partial \overline{z_{j}}}\frac{\partial}{\partial \overline{z_{n}}}\right),  1 \leq j \leq n-1.
\end{equation}

Then $\{L_{j}\}_{1 \leq j \leq n-1}$ forms a basis for the CR vector fields along $M$ near $p.$ By applying $\overline{L_{j}},  \overline{L_{j}L_{k}},  1 \leq j, k \leq n-1$ to the equation (\ref{eqn}) and evaluating at $0$,   we get:

$$\frac{\partial F_{n+1}}{\partial z_{j}}(0)=0,  \frac{\partial^2 F_{n+1}}{\partial z_{j}\partial z_{k}}(0)=0,  1 \leq j,  k \leq n-1.$$

Hence we have, 
\begin{equation}\label{eqn1a}
F_{n+1}(Z)=\lambda z_{n} + O(|Z|^2).
\end{equation}
For $1 \leq j \leq n$,   we write
\begin{equation}
F_{j}=a_{j}z_{n}+ \sum_{i=1}^{n-1}a_{ij}z_{i}+O(|Z|^2), 
\end{equation}
for some $a_{j} \in \mathbb{C},  a_{ij} \in \mathbb{C},  1 \leq i \leq n-1,  1 \leq j \leq n.$ Or equivalently, 
\begin{equation}\label{eqn2}
(F_{1}, ..., F_{n})=z_{n}(a_{1}, ..., a_{n})+ (z_{1}, ..., z_{n-1})A + (\hat{F}_{1}, ..., \hat{F}_{n}), 
\end{equation}
where $A=(a_{ij})_{1 \leq i \leq n-1,  1 \leq j \leq n}$ is an $(n-1) \times n$ matrix,  and
$\hat{F}_{j}=O(|Z|^2),  1 \leq j \leq n.$
We plug in (\ref{eqn1a}) and (\ref{eqn2}) into (\ref{eqn}) to get, 
\begin{equation}\label{eqn3}
\lambda |\widetilde{Z}|^2 + O(|\widetilde{Z}||z_{n}| + |z_{n}|^2)+o(|Z|^2)
=\widetilde{z}AUA^{*}\overline{\widetilde{z}}^t + O(|\widetilde{Z}||z_{n}|+|z_{n}|^2)
+o(|Z|^2), 
\end{equation}
where we write $\widetilde{Z}=(z_{1}, ..., z_{n-1}).$ Equip $\widetilde{Z}$ with weight $1$,   and $z_{n}$ with weight $2.$ We then compare terms with weight $2$ at both sides of (\ref{eqn3}) to get:
\begin{equation}\label{eqnu}
\lambda {\bf I}_{n-1}=AUA^{*}.
\end{equation}

As a consequence,  the matrix $A$ has full rank $(n-1),$ and  $U$ has rank $(n-1)$ or $n.$ Recall
that $M'$ is not $1-$nondegenerate at $q=0.$ We thus conclude that $U$ has rank $(n-1).$
Moreover,  note from (\ref{eqnu}) that $U$ has $(n-1)$ positive eigenvalues.  By a holomorphic change of coordinates in $\mathbb{C}^{n+1}$,   we may assume that
$U=\mathrm{diag}\{1, ..., 1, 0\}.~M'$ is then of the following form near $0:$
\begin{equation}\label{eqn4}
\rho(W, \overline{W})=-\mathrm{Im}w_{n+1} +\sum_{j=1}^{n-1}|w_{j}|^2 + \phi(W, \overline{W}), \quad \phi=O(|W|^3).
\end{equation}

Write $A=(B, {\bf b})$,   where $B$ is a $(n-1) \times (n-1)$ matrix,  $b$ is an $(n-1)-$dimensional column vector. (\ref{eqnu}) yields that $B\overline{B}^t=\lambda {\bf I}_{n-1}.$ We now apply the following holomorphic change of coordinates: $\widetilde{W}=WD$ or $W=\widetilde{W}D^{-1}$,   where we set
$$D=\left(
      \begin{array}{ccc}
        \frac{1}{\sqrt{\lambda}}\overline{B}^t & {\bf c} & {\bf 0} \\
        {\bf 0}^t & 1 & 0 \\
        {\bf 0}^t & 0 & 1 \\
      \end{array}
    \right), 
$$
and ${\bf 0}$ is the $(n-1)-$dimensinal zero column vector,  ${\bf c}$ is a $(n-1)-$dimensional column vector to be determined. We compute
$$D^{-1}=\left(
\begin{array}{ccc}
\frac{1}{\sqrt{\lambda}}B & {\bf d} & {\bf 0} \\
{\bf 0}^t & 1 & 0 \\
{\bf 0}^t & 0 & 1 \\
\end{array}
\right),$$ 
where ${\bf d}=-\frac{1}{\sqrt{\lambda}}B{\bf c}.$

We write the new defining function of $M'$ and the map as $\widetilde{\rho}$ and
$\widetilde{F}=(\widetilde{F}_{1}, ..., \widetilde{F}_{n+1})$ in the new coordinates
$\widetilde{W}=(\widetilde{w}_{1}, ..., \widetilde{w}_{n+1})$,   respectively. We have

\begin{lemma} \label{lemma1}
 $\widetilde{\rho}$ still has the form of (\ref{eqn4}). More precisely, 
$$\widetilde{\rho}(\widetilde{W}, \overline{\widetilde{W}})=-\mathrm{Im}\widetilde{w}_{n+1}
+\sum_{j=1}^{n-1}|\widetilde{w}_{j}|^2 +\widetilde{\phi}(\widetilde{W}, \overline{\widetilde{W}}),$$ 
where $\widetilde{\phi}(\widetilde{W},  \overline{\widetilde{W}})=O(|\widetilde{W}|^3)$ is also a real-analytic (resp. smooth) function defined near $0.$
\end{lemma}

\begin{proof} This can be checked by a simple calculation and using the fact that
$$\left(
    \begin{array}{cc}
      \frac{1}{\sqrt{\lambda}}B & {\bf d} \\
      {\bf 0}^t & 1 \\
    \end{array}
  \right)\left(
           \begin{array}{cc}
             {\bf I}_{n} & {\bf 0} \\
             {\bf 0}^t & 0 \\
           \end{array}
         \right)\left(
                  \begin{array}{cc}
                    \frac{1}{\sqrt{\lambda}}\overline{B}^t & {\bf 0}  \\
                    \overline{\bf d}^t & 1 \\
                  \end{array}
                \right)
         =\left(
                   \begin{array}{cc}
                     {\bf I}_{n} & {\bf 0} \\
                     {\bf 0}^t & 0 \\
                   \end{array}
                 \right).
$$
\end{proof}
\bigskip

Moreover,  since $\widetilde{F}=FD$,    it is easy to see that
$$\frac{\partial \widetilde{F}_{i}}{\partial z_{j}}(0)=\delta_{ij}\sqrt{\lambda}, ~1 \leq i,  j \leq n-1.$$ Here we denote by $\delta_{ij}$  the Kronecker symbol that takes value $1$ when $i=j$ and $0$ otherwise.

\begin{lemma} \label{lemma2} We can choose an appropriate ${\bf c}$ such that
\begin{equation}
\frac{\partial \widetilde{F}_{n}}{\partial z_{j}}(0)=0,  1 \leq j \leq n-1.
\end{equation}

{\bf Proof:} Note that $\widetilde{F}_{n}=(F_{1}, ..., F_{n})\left(
                                                             \begin{array}{c}
                                                               {\bf c} \\
                                                               1 \\
                                                             \end{array}
                                                           \right).
$
Combining this with (\ref{eqn2}),  we obtain, 
$$\frac{\partial \widetilde{F}_{n}}{\partial z_{j}}(0)=0,  1 \leq j \leq n-1$$
is equivalent to $A\left(
                                                             \begin{array}{c}
                                                               {\bf c} \\
                                                               1 \\
                                                             \end{array}
                                                           \right)={\bf 0}$,  
where ${\bf 0}$ is the $(n-1)-$dimensinal zero column vector. Recall $A=(B,  {\bf b}).$ We can thus choose ${\bf c}=-B^{-1}{\bf b}.$
\end{lemma}

In the following,  for brevity,  we still write $W,  F$ and $\rho$ instead of $\widetilde{W},  \widetilde{F}$ and $\widetilde{\rho}.$ We summarize
the considerations of this section in the following

\begin{proposition}\label{rknormalization}
Let $M \subset \mathbb{C}^n (n \geq 2)$ be a strongly pseudoconvex real-analytic (resp. smooth) real hypersurface,  $M' \subset \mathbb{C}^{n+1}$ a  real-analytic (resp. smooth) real hypersurface. Assume that $F=(F_{1}, ..., F_{n+1}): M \mapsto M'$ is a CR-transversal CR-mapping of class $C^2$ near $p_{0} \in M$ with $F(p_{0})=q_{0}$,  and that $M'$ is Levi-degenerate at $q_0$.  Then,  after  appropriate holomorphic changes of coordinates in $\mathbb{C}^n$ and $\mathbb{C}^{n+1}$ respectively,  we have $p_{0}=0,  q_{0}=0$,  and the following normalizations hold. $M$ is defined by
\begin{equation}
r(Z, \overline{Z})=-\mathrm{Im}z_{n}+ \sum_{i=1}^{n-1}|z_i|^2+\psi(Z, \overline{Z}),  \quad \psi=O(|Z|^3)
\end{equation}
near $0$,  and $M'$ is defined by
\begin{equation}\label{eqnrho}
\rho(W, \overline{W})=-\mathrm{Im}w_{n+1} +\sum_{j=1}^{n-1}|w_{j}|^2 + \phi(W, \overline{W}),  \quad \phi=O(|W|^3)
\end{equation}
near $0$,   where $Z=(z_{1}, ..., z_{n}),  W=(w_{1}, ..., w_{n+1})$ are the coordinates of $\mathbb{C}^n$ and $\mathbb{C}^{n+1}$,   respectively. Furthermore,  $F$ satisfies:
\begin{equation}
\frac{\partial F_{i}}{\partial z_{j}}(0)=\delta_{ij}\sqrt{\lambda}, ~1 \leq i,  j \leq n-1, 
\end{equation}
for some $\lambda > 0$,   and moreover, 
\begin{equation}
\frac{\partial F_{n}}{\partial z_{j}}(0)=0, \quad 1 \leq j \leq n-1;
\end{equation}
\begin{equation}
\frac{\partial F_{n+1}}{\partial z_{j}}(0)=0,  \quad 1 \leq j \leq n-1.
\end{equation}

\end{proposition}

One can use the same arguments as given above to arrive at the following, 
more general conclusion, as pointed out by one of the anonymous referees; 
we record it here, even though we only need the version given above in 
this paper. In particular, the following formulation recovers as a special case ($m = n + \ell -1$) a result of Berhanu and the third author \cite[Lemma 4.2]{BX1}.

\begin{proposition}\label{rknormalization2} Let 
$M \subset \mathbb{C}^n (n \geq 2)$ 
be a strongly pseudoconvex real-analytic (resp. smooth) real hypersurface,
$M' \subset \mathbb{C}^{n+\ell}$ a real-analytic (resp. smooth) real hypersurface.
Assume that $F=(F_{1}, ..., F_{n+\ell}): M \mapsto M'$ is a CR-transversal CR 
mapping of class $C^2$ near $p_{0} \in M$ with $F(p_{0})=q_{0}$,  and that the Levi-form of $M'$ has, say $m$ nonzero eigenvalues, all of the same sign.  
Then,  after  appropriate holomorphic changes of
coordinates in $\mathbb{C}^n$ and $\mathbb{C}^{n+\ell}$ respectively,  we have
$p_{0}=0,  q_{0}=0$,  and the following normalizations hold. $M$ is defined by
\begin{equation} r(Z, \overline{Z})=-\mathrm{Im}z_{n}+
\sum_{i=1}^{n-1}|z_i|^2+\psi(Z, \overline{Z}),  \quad \psi=O(|Z|^3)
\end{equation} near $0$,  and $M'$ is defined by \begin{equation}\label{eqnrho2}
\rho(W, \overline{W})=-\mathrm{Im}w_{n+\ell} +\sum_{j=1}^{m}|w_{j}|^2 + \phi(W,
\overline{W}),  \quad \phi=O(|W|^3) \end{equation} near $0$,   where $Z=(z_{1},
..., z_{n}),  W=(w_{1}, ..., w_{n+\ell})$ are the coordinates of $\mathbb{C}^n$ and
$\mathbb{C}^{n+\ell}$,   respectively. Furthermore,  $F$ satisfies:
\begin{equation} \frac{\partial F_{i}}{\partial
z_{j}}(0)=\delta_{ij}\sqrt{\lambda}, ~1 \leq i,  j \leq n-1,  \end{equation} for
some $\lambda > 0$,   and moreover,  \begin{equation} \frac{\partial
F_{n+k}}{\partial z_{j}}(0)=0,  \quad  1 \leq j \leq n-1, \quad  1\leq k \leq \ell. \end{equation}
\end{proposition}

In order to prove this more general assertion, one just has to follow 
the steps of the proof of \autoref{rknormalization} taking 
into account the more general codimension. This leads to a number 
of scalar quantities quantities to be vectors, but does not impact 
the main argument. 

\section{Proof of Theorem 1}
In this section will make use of the normalization of $M, M'$ in Proposition 3.3 and prove \autoref{T0}. We will see that in the setting of \autoref{T0}, the map is 
actually $2$-nondegenerate. The result then follows in the smooth case
by applying \cite[Theorem 2]{L1}  and in the real-analytic case by applying \cite[Theorem 1.3]{L2}. 
Our first step therefore 
is the following  computational Lemma for the uniformly $2-$nondegenerate target hyperurface $M'$. 
For further results about normal forms along this line,  see \cite{E1}. 
\begin{lemma}
\label{lem:revision} Let $M \subset \mathbb{C}^n\,  (n \geq 2)$ be a strongly pseudoconvex real-analytic (resp. smooth)  hypersurface,  and $M' \subset \mathbb{C}^{n+1}$ a uniformly
$2-$nondegenerate real-analytic (resp. smooth)  hypersurface. Assume that $F=(F_{1}, ..., F_{n+1}): M \mapsto M'$ is a CR-transversal CR-mapping of class $C^2$. Then $F$ is 
$2$-nondegenerate.  
\end{lemma}
\begin{proof}
We will write for $1 \leq k \leq n,$
\begin{equation}
\Lambda_{k}=2i\left(\frac{\partial \rho}{\partial \overline{w}_{n+1}}\frac{\partial}{\partial \overline{w}_{k}}- \frac{\partial \rho}{\partial \overline{w}_{k}}\frac{\partial}{\partial \overline{w}_{n+1}}\right), 
\end{equation}
where $\{\Lambda_{k}\}_{1 \leq k \leq n}$ forms a basis for the CR vector fields along $M'$ near 0.  Note that
\begin{equation}
\begin{split}
\Lambda_{k} & =(1+2i\phi_{\overline{n+1}})\frac{\partial}{\partial \overline{w}_{k}}
-2i(w_{k}+\phi_{\overline{k}})\frac{\partial}{\partial \overline{w}_{n+1}}, ~\text{if}~
1 \leq k \leq n-1, \\
\Lambda_{n} & =(1+ 2i \phi_{\overline{n+1}})\frac{\partial}{\partial \overline{w}_{n}}-
2i(\phi_{\overline{n}})\frac{\partial}{\partial \overline{w}_{n+1}}.
\end{split}
\end{equation}
Here and in the following,   we write for $1 \leq i,  j,  k \leq n+1$,   $\phi_{\overline{i}}=\phi_{\overline{w}_i}=\frac{\partial \phi}{\partial \overline{w}_i},  \phi_{i}=\phi_{w_i}=\frac{\partial \phi}{\partial w_i},  \phi_{i\overline{j}}=\phi_{w_i\overline{w}_j}=\frac{\partial^2 \phi}{\partial w_i \partial \overline{w}_j},  \phi_{\overline{ij}k}=\phi_{\overline{w}_i\overline{w}_j w_k}=\frac{\partial^3 \phi}{\partial \overline{w}_i \partial \overline{w}_j \partial w_k}$,   etc. 

Recall our notation $\rho_{W}:=(\frac{\partial \rho}{\partial w_1}, ...,  \frac{\partial \rho}{\partial w_{n+1}}).$ We compute
\begin{equation}\label{eqnrho43}
\rho_{W}(W, \overline{W})=(\overline{w}_{1}+\phi_{1}, ..., \overline{w}_{n-1}+\phi_{n-1},  \phi_{n},  \frac{i}{2}+\phi_{n+1})\\
\end{equation}
We thus have
\begin{equation}
\Lambda_{1}\rho_{W}(W, \overline{W})=\left(h_{11}, ..., h_{1(n+1)} \right), 
\end{equation}
where
\begin{equation}
\begin{split}
h_{11}&= (1+2i\phi_{\overline{(n+1)}})(1+\phi_{1\overline{1}})
-2i(w_{1}+\phi_{\overline{1}})\phi_{1\overline{(n+1)}}, \\
h_{12}&= (1+2i\phi_{\overline{(n+1)}})\phi_{2\overline{1}}
-2i(w_{1}+\phi_{\overline{1}})\phi_{2\overline{(n+1)}}, \\
\cdots, \\
h_{1(n+1)}&=(1+2i\phi_{\overline{(n+1)}})\phi_{(n+1)\overline{1}}-
2i(w_{1}+\phi_{\overline{1}})\phi_{(n+1)\overline{(n+1)}}.\\
\end{split}
\end{equation}
Hence
\begin{equation}
\Lambda_{1}\rho_{W}(W, \overline{W})=(1+O(1), O(1), ..., O(1)).
\end{equation}
Here we write $O(m)=O(|W|^m)$ for any $m \geq 0.$
Similarly, we have for $1 \leq k \leq n-1$,  
\begin{equation}\label{eqnrho47}
\Lambda_{k}\rho_{W}(W, \overline{W})=(O(1), ..., O(1), 1+O(1), O(1), ..., O(1)), 
\end{equation}
where the term $1+O(1)$ is at the $k^{\text{th}}$ position;
\begin{equation}
\Lambda_{n}\rho_{W}(W,  \overline{W})=(O(1), ..., O(1), \phi_{n\overline{n}}+O(2),  O(1)).
\end{equation}
As a consequence,  we have
\begin{equation}\label{1dege}
\mathrm{det}\left(
                    \begin{array}{c}
                      \rho_{W}(W,  \overline{W}) \\
                      \Lambda_{1}\rho_{W}(W, \overline{W}) \\
                      \cdots \\
                      \Lambda_{n}\rho_{W}(W, \overline{W}) \\
                    \end{array}
                  \right)=\pm\frac{i}{2}\phi_{n\overline{n}}+O(2).
\end{equation}

Recall that $M'$ is uniformly $2-$nondegenerate at $0$,   in particular,  it is $1-$degenerate at every point near $0.$ This implies (\ref{1dege}) is identically zero near $0$ along $M'.$ Consequently,  by applying $\Lambda_{j},  1 \leq j \leq n$ to (\ref{1dege}) and evaluating at $0$,    we obtain $\phi_{\overline{jn}n}(0)=0$ for any $1 \leq j \leq n.$

Since $M'$ is  $2$-nondegenerate, there exists 
a choice $j_0, k_0$, $1\leq j_0, k_0 \leq n$ such that  $\Lambda_{j_0} \Lambda_{k_0} \rho_W (W,\bar W)$
together with the $\Lambda_j \rho_W (W,\bar W)$ and $\rho_W (W, \bar W)$ 
spans $\C^{n+1}$ for $W$ close to $0$. Hence, for that 
choice of $j_0, k_0$, 
 we have:
\begin{equation}
\phi_{\overline{j_{0}k_{0}}n}(0) \neq 0,  ~\text{for some}~1 \leq j_{0},  k_{0} \leq n-1.
\end{equation}
Consequently,  if we write
\begin{equation}\label{eqnljlk}
L_{j_{0}}L_{k_{0}}\rho_{W}(F, \overline{F})(0):=(\nu_{1}, ..., \nu_{n-1}, \nu_{n}, \nu_{n+1}), 
\end{equation}
then $\nu_{n}$ is nonzero. Here $L_j$ is as defined in (\ref{eqnlj}). Indeed, 

$$\nu_n = \frac{\partial^2 \phi_n (F,  \overline{F})}{\partial \overline{z}_{j_0} \partial \overline{z}_{k_0}}\biggr|_{0}=\sum_{i, j=1}^{n+1}\frac{\partial^2 \phi_n}{\partial \overline{w}_i \partial \overline{w}_j}\biggr|_0 \overline{\frac{\partial F_i}{\partial z_{j_0}}}\biggr|_{0} \overline{\frac{\partial F_j}{\partial z_{k_0}}}\biggr|_{0}=\phi_{\overline{j_0k_0}n}(0) \overline{\frac{\partial F_{j_0}}{\partial z_{j_0}}}\biggr|_{0} \overline{\frac{\partial F_{k_0}}{\partial z_{k_0}}}\biggr|_{0} \neq 0.$$

Moreover,  it is easy to verify that
\begin{equation}\label{eqnli}
L_{i}\rho_{W}(F, \overline{F})(0)=(0, ..., 0, \sqrt{\lambda}, 0, ..., 0),  1 \leq i \leq n-1, 
\end{equation}
where $\sqrt{\lambda}$ is at the $i^{\rm th}$ position,  and that
\begin{equation}\label{eqnrw}
\rho_{W}(F, \overline{F})(0)=(0, ..., 0, \frac{i}{2}).
\end{equation}

Equations (\ref{eqnrw}),  (\ref{eqnli}) and (\ref{eqnljlk}) with $\nu_n \neq 0$ now imply that $F$ is $2-$nondegenerate at $0$.
\end{proof} 

\begin{proof}[Proof of Theorem 1] By the above-mentioned results of \cite{L1, L2},  $F$ is real-analytic (resp. smooth) near $0$,  as required.
\end{proof}

\section{Proof of Theorem 2}
Let $M'$ be as above and $\rho$ as in (\ref{eqnrho}).  For any $1 \leq i_1 \leq \cdots \leq i_l \leq n, q \in M',$ we define,

$$\Delta_{i_1...i_l}(q)= \det \left(
    \begin{array}{c}
      \rho_{W} \\
      \Lambda_{1}\rho_{W} \\
      ... \\
      \Lambda_{n-1}\rho_{W} \\
      \Lambda_{i_{1}}...\Lambda_{i_{l}}\rho_{W} \\
    \end{array}
  \right)(q).$$

We first prove the following lemma.
\begin{lemma}\label{lemma51}
Let $ M'$ be as above. Assume that $M'$ is $l$-nondegenerate at $0$ for some $l \geq 2.$ Then there exist $1 \leq i_{1} \leq ...\leq i_{l} \leq n$,    such that
\begin{equation}\label{eqndelta}
\Delta_{i_1...i_l}(0) \neq 0.
\end{equation}
\end{lemma}
\begin{proof}
We note that
\begin{equation}\label{eqn52rw}
\rho_{W}(0)=(0, ..., 0,   \frac{i}{2}),
\end{equation}
\begin{equation}\label{eqn53dw}
\Lambda_{j}\rho_{W}(0)=(0, .., 0, 1, 0, ..., 0), 1 \leq j \leq n-1, 
\end{equation}
where $1$ is at the $j^{\rm th}$ position. Thus $\rho_W(0), \Lambda_j\rho_W(0), 1 \leq j \leq n-1,$ are linearly independent over $\mathbb{C}.$ 
Then by the definition of $l$-nondegeneracy at $0$,  one easily sees that  there exists $1 \leq i_{1} \leq ...\leq i_{l} \leq n$ such that
$$ \det \left(
    \begin{array}{c}
      \rho_{W} \\
      \Lambda_{1}\rho_{W} \\
      ... \\
      \Lambda_{n-1}\rho_{W} \\
      \Lambda_{i_{1}}...\Lambda_{i_{l}}\rho_{W} \\
    \end{array}
  \right)(0) \neq 0.
$$
\end{proof}
\begin{remark}\label{remark51}
In particular, when $l=2$ in Lemma \ref{lemma51}, there exist $1 \leq i_1 \leq i_2 \leq n,$ such that $\Delta_{i_1i_2}(0) \neq 0.$ Note  the $n^{\text{th}}$ component of $\Lambda_{i_{1}}\Lambda_{i_2}\rho_{W} (0)$ is $\phi_{\overline{i_1i_2}n}(0).$
By the form (\ref{eqn52rw}), (\ref{eqn53dw}) of $\rho_W(0)$ and $\Lambda_j\rho_W(0),$ we conclude that $\phi_{\overline{i_1i_2}n}(0) \neq 0.$
\end{remark}

\bigskip

We then prove the following proposition.
\begin{proposition}\label{thml2}
Let $M \subset \mathbb{C}^n (n \geq 2)$ be a strongly pseudoconvex  real-analytic (resp. smooth)  hypersurface,  and
$M' \subset \mathbb{C}^{n+1}$ be a   real-analytic (resp. smooth)  hypersurface. Assume that $M'$ is either $1$- or  $2$-nondegenerate at  every point of it. Let $F=(F_{1}, ..., F_{n+1}): M \mapsto M'$ be a CR-transversal CR-mapping of class $C^2.$ Then $F$ is real-analytic (resp. smooth) on
a dense open subset of $M.$
\end{proposition}

\begin{proof}
We write $\Omega$ as the open subset of $M$ where $F$ is real-analytic (resp. smooth).
Fix any $p_{0} \in M.$ Write $q_{0}=F(p_{0}) \in M'.$ We will need to prove $p_0 \in \overline{\Omega}.$ We assume $p_{0}=0 \in M,  q_{0}=0 \in M'.$ By assumption,  $M'$ is either $1$-nondegenerate or $2$-nondegenerate at $q_{0}.$ We then split our argument in two cases.

{\bf Case I:}  $M'$ is $1$-nondegenerate at $q_{0}.$ That is,  $M'$ is Levi-nondegenerate near $q_{0}.$ Then it follows from Corollary 2.3 in [BX2] that $p_{0} \in \overline{\Omega}.$

{\bf Case II:} $M'$ is $2$-nondegenerate at $q_{0}.$ Let $O$ be a small neighborhood of $q_{0}$ in $\mathbb{C}^{n+1}.$ Let $V=O \cap M'.$  We write $V_{1}$ as the set of $1$-degeneracy of $M'$ in $V.$ More precisely, 
$$V_{1}=\{q \in V: M'~\text{is}~1\text{-degenerate at}~q \}.$$

 If there is a sequence $\{p_{i}\}_{i=1}^{\infty} \subset M$ converging to $p_{0}$ such that $M'$ is $1-$nondegenerate at each $F(p_{i})$,   i.e.,  $F(p_{i}) \in M \setminus V_{1},   i \geq 1.$ Then by Case I,  we have each $p_{i} \in \overline{\Omega},  i \geq 1.$ Consequently,  $p_{0} \in \overline{\Omega}.$ Thus we are only left with the case
 that there exists a neighborhood $U$ of $p$ on $M$ such that $F(U) \subset V_{1}.$  We apply then normalization to $M,  M'$ and the map $F$ as in \autoref{rknormalization}. Since $M'$ is $2$-nondegenerate at $0$,   we conclude by \autoref{lemma51},  $\Delta_{j_0k_0}(0)=c \neq 0$,   for some $1 \leq j_{0} \leq k_{0} \leq n.$ We then further split into the following subcases.

{\bf Case II(a)}: There exist some $1 \leq j_{0} \leq k_{0} \leq n-1$,   such that, 
$\Delta_{j_0k_0}(0)=c \neq 0.$  Consequently, we have $\phi_{\overline{j_0k_0}n}(0) \neq 0.$ Then similarly as in the proof of \autoref{T0}, 
we can show that $F$ is finitely nondegenerate.  Hence again by the results of \cite{L1, L2},  $F$
is real-analytic (resp. smooth) at $0.$

{\bf Case II(b)}: For any $1 \leq j \leq k \leq n-1$,   $\Delta_{jk}(0)=0.$ Moreover,  there exists $1 \leq j_{0} \leq n-1$ such that,  $\Delta_{j_0n}(0)=c \neq 0.$ Then by a similar argument as in Remark \ref{remark51}, we conclude $\phi_{\overline{jk}n}(0)=0,$ for any $1 \leq j \leq k \leq n-1,$ and $\phi_{\overline{j_0n}n}(0) \neq 0.$

Note that $V_{1} \subset \widetilde{V}_{1}\bigcap M'$,   where $\widetilde{V}_{1}$ is defined
\begin{equation}\label{eqnv51}
\widetilde{V}_{1}:=\{W \in O: \varphi(W,  \overline{W})=0 \}, 
\end{equation}
with
\begin{equation}\label{eqnphi52}
\varphi(W, \overline{W})={\rm det}\left(
                                   \begin{array}{c}
                                     \rho_{W} \\
                                     \Lambda_{1}\rho_{W} \\
                                     ... \\
                                     \Lambda_{n-1}\rho_{W} \\
                                     \Lambda_{n}\rho_{W} \\
                                   \end{array}
                                 \right)(W,  \overline{W}).
\end{equation}

Then we have

\begin{claim}\label{lmwj0}
The $\overline{w}_{j_{0}}$-derivative of $\varphi$ is nonzero at $q_{0}=0.$
\end{claim}
\begin{proof}
It is equivalent to show that $\Lambda_{j_{0}}\varphi(0) \neq 0.$ That is, 
\begin{equation}
\Lambda_{j_{0}}\det\left(
                                   \begin{array}{c}
                                     \rho_{W} \\
                                     \Lambda_{1}\rho_{W} \\
                                     ... \\
                                     \Lambda_{n-1}\rho_{W} \\
                                     \Lambda_{n}\rho_{W} \\
                                   \end{array}
                                 \right)(0) \neq 0.
\end{equation}

Note that
$$\Lambda_{j_{0}}\det\left(
                                   \begin{array}{c}
                                     \rho_{W} \\
                                     \Lambda_{1}\rho_{W} \\
                                     ... \\
                                     \Lambda_{n-1}\rho_{W} \\
                                     \Lambda_{n}\rho_{W} \\
                                   \end{array}
                                 \right)(0)=$$

$$\det\left(
                                   \begin{array}{c}
                                     \Lambda_{j_{0}}\rho_{W} \\
                                     \Lambda_{1}\rho_{W} \\
                                     ... \\
                                     \Lambda_{n-1}\rho_{W} \\
                                     \Lambda_{n}\rho_{W} \\
                                   \end{array}
                                 \right)(0)+ \det\left(
                                   \begin{array}{c}
                                     \rho_{W} \\
                                     \Lambda_{j_{0}}\Lambda_{1}\rho_{W} \\
                                     ... \\
                                     \Lambda_{n-1}\rho_{W} \\
                                     \Lambda_{n}\rho_{W} \\
                                   \end{array}
                                 \right)(0)+...+\det\left(
                                   \begin{array}{c}
                                     \rho_{W} \\
                                     \Lambda_{1}\rho_{W} \\
                                     ... \\
                                     \Lambda_{j_{0}}\Lambda_{n-1}\rho_{W} \\
                                     \Lambda_{n}\rho_{W} \\
                                   \end{array}
                                 \right)(0)+\det\left(
                                   \begin{array}{c}
                                     \rho_{W} \\
                                     \Lambda_{1}\rho_{W} \\
                                     ... \\
                                     \Lambda_{n-1}\rho_{W} \\
                                     \Lambda_{j_{0}}\Lambda_{n}\rho_{W} \\
                                   \end{array}
                                 \right)(0).$$

In the above equation,  the first term is trivially zero. Then we note that in the row
vector $\rho_{W}(0)$,   or $\Lambda_{i}\rho_{W}(0),  1 \leq i \leq n$,   the $n^{\rm th}$ component
is zero. This is due to the fact that $\phi=O(|W|^3).$ 
Moreover,  the $n^{\rm th}$ component in the row vector $\Lambda_{j_{0}}\Lambda_{k}\rho_{W}(0),  1 \leq k \leq n-1$,   is $\phi_{\overline{j_0k}n}(0),$ which is zero by the assumption. Consequently,  the second term up to 
the $n^{\rm th}$ term in the above equation are all zero. 
We also note the last term in the equation above is just equal to $\Delta_{j_0n}(0),$ which is nonzero.
Hence  the lemma is established.
\end{proof}

Recall that $F(U) \subset V_{1} \subset \widetilde{V}_{1}.$ We have
\begin{equation}
\varphi(F(Z), \overline{F(Z)}) \equiv 0, ~\text{for all}~Z \in U \subset M.
\end{equation}
Applying $L_{j_{0}}$ to the above equation and evaluating at $Z=0$, we have, 
\begin{equation}\label{eqnlphifzero}
L_{j_{0}}\varphi(F, \overline{F})|_{0}=\sum_{i=1}^{n+1}\varphi_{\overline{w}_{i}}
(F, \overline{F})|_{0}L_{j_{0}}\overline{F}_{i}|_{0}=0.
\end{equation}

Note that by our normalization,   $L_{j_{0}}\overline{F}_{i}(0)=0$,   if $i \neq j_{0}.~ L_{j_{0}}\overline{F}_{j_{0}}(0) \neq 0.$ Moreover,  by \autoref{lmwj0},  $\varphi_{\overline{w}_{j_{0}}}(0) \neq 0.$ This is a contradiction to (\ref{eqnlphifzero}).
Hence {\bf Case II(b)} cannot happen in this setting.

\bigskip

{\bf Case II(c)}: $\Delta_{jk}(0)=\Delta_{jn}(0)=0$,   for all $1 \leq j, k \leq n-1$,   and $\Delta_{nn}(0) \neq 0.$

We let $\widetilde{V}_{1}$ be defined by $\varphi$ as above in (\ref{eqnv51}),  (\ref{eqnphi52}).
\begin{claim} \label{lmwnqn}
In the setting of this subcase,  the $\bar{w}_{n}$-derivative of $\varphi$ is nonzero at $q_{0}=0.$
\end{claim}
\begin{proof}
Similar as Lemma \autoref{lmwj0}.
\end{proof}

By Lemma \autoref{lmwnqn},  we have $\varphi_{\overline{w}_{n}}(0) \neq 0.$ Consequently,  if
we define $\widetilde{\varphi}(W, \overline{W})=\overline{\varphi(W, \overline{W})}$,   then
\begin{equation}\label{eqnphi58}
\widetilde{\varphi}_{w_{n}}(0) \neq 0.
\end{equation}

Note $F(U) \subset V_{1} \subset \widetilde{V}_{1}.$ We have
$$\varphi(F(Z), \overline{F(Z)}) \equiv 0, ~\text{for all}~Z \in U.$$
Consequently, 
\begin{equation}\label{eqnphif}
\widetilde{\varphi}(F(Z),  \overline{F(Z)}) \equiv 0.
\end{equation}
Recall that for all $Z \in U$,  
\begin{equation}\label{eqnrhof}
\rho(F, \overline{F})=0, 
\end{equation}

\begin{equation}\label{eqnlrhof}
L_{i}\rho(F, \overline{F})=0,  1 \leq i \leq n-1.
\end{equation}

Combining (\ref{eqnli}),  (\ref{eqnrw}),  (\ref{eqnphi58}),  we conclude that
\begin{equation}\label{eqnnondege}
\det\left(
  \begin{array}{c}
    \rho_{W}(F,  \overline{F}) \\
    L_{1}\rho_{W}(F,  \overline{F}) \\
    ... \\
    L_{n-1}\rho_{W}(F,  \overline{F}) \\
    \widetilde{\varphi}_{W}(F, \overline{F}) \\
  \end{array}
\right) (0) \neq 0.
\end{equation}
This implies that equations (\ref{eqnphif}),  (\ref{eqnrhof}), (\ref{eqnlrhof}) form a nondegenerate system for $F.$ Then it follows that $F$ is real-analytic (resp. smooth) at $0$ by a similar argument as in \cite{L1, L2} or \cite{BX1, BX2}. For the convenience of the readers,  we sketch a proof here for the real-analytic category. The proof for 
the smooth category is similar (we remark on the differences). However, 
those readers unfamiliar with the strategy, we refer to the aforementioned papers for the 
necessary details. 

  We assume that $M$ is defined near 
$0$ by $\{(z, z_n)=(z, s+it) \in U_0 \times V: t=\phi(z,  \overline{z},  s) \}$,   where $\phi$ is a real-valued,  real-analytic function
with $\phi(0)=0,  d\phi(0)=0.$ Here  $U_0 \subset \mathbb{C}^{n-1}$ and $V \subset \mathbb{R}$ are sufficiently small open subsets. In the local coordinates $(z, s) \in \mathbb{C}^{n-1} \times \mathbb{R}$,   we may assume 
that, 
$$L_j=\frac{\partial}{\partial \overline{z}_j}-i\frac{\phi_{\overline{z}_j}(z,  \overline{z},  s)}{1+i\phi_{s}(z,  \overline{z},  s)}\frac{\partial}{\partial s},  1 \leq j \leq n-1.$$
Since $\phi$ is real-analytic,  we can complexify 
(resp. extend almost holomorphically, in the smooth case) in the $s$ variable and write

$$M_j=\frac{\partial}{\partial \overline{z}_j}-i\frac{\phi_{\overline{z}_j}(z,  \overline{z},  s+it)}{1+i\phi_{s}(z,  \overline{z},  s+it)}\frac{\partial}{\partial s},  1 \leq j \leq n-1,$$ 
which are holomorphic in $s+it$ and extend the vector fields $L_j.$

Since $\phi$ and $L_j$ are real-analytic now,  equations (\ref{eqnphif}),  (\ref{eqnrhof}),  (\ref{eqnlrhof}) imply that there is real-analytic map $\Phi(W,  \overline{W},  \Theta)$ defined in a neighborhood of $\{ 0 \} \times \mathbb{C}^q$ in $\mathbb{C}^{n+1} \times \mathbb{C}^q$,   polynomial in the last $q$ variables for some integer $q$ such that 
$$\Phi(F,  \overline{F},  (L^{\alpha} \overline{F})_{|\alpha|=1})=0$$
at  $(z, s) \in U_0 \times V.$ By (\ref{eqnnondege}) the matrix 
$\Phi_{W}$ is invertible at the central point $0$,   by the holomorphic version
of the implicit function theorem. (In the smooth category,  we apply the ``almost holomorphic" version of the implicit function theorem,  cf. \cite{L1}).  We get a 
holomorphic map $\Psi=(\Psi_1, ..., \Psi_{n+1})$ such that for $(z, s)$ near the origin, 
$$F_j=\Psi_j(\overline{F},  (L^{\alpha}\overline{F})_{|\alpha|=1}),  1 \leq j \leq n+1.$$

We now set for each $1 \leq j \leq n+1$,  
$$h_j(z, s, t)=\Psi_{j}\left(\overline{F}(z, s, -t),  (M^{\alpha} \overline{F})_{|\alpha|=1}(z, s, -t)\right).$$
Since $M$ is strongly pseudoconvex,  the CR functions $F_j,  1 \leq j \leq n+1$,   all extend as holomorphic functions in $s+it$ to the side $t >0.$
Hence the conjugates $\overline{F}_j,  1 \leq j \leq n+1$,   extend holomorphically to the side $t <0.$ It now follows that $F_j,  1 \leq j \leq n+1$,   extend as holomorphic functions to a full neighborhood of the origin
(See Lemma 9.2.9 in \cite{BER}). (In the smooth category, we apply the 
edge-of-the-wedge theorem, as in \cite{L1}). This establishes Proposition \ref{thml2}.

\end{proof}

We then prove \autoref{T1}.

{\bf Proof of \autoref{T1}:} We again write $\Omega$ for the open subset of $M$ where $F$ is smooth.
Fix any $p_{0} \in M$ and let $q_{0}=F(p_{0}) \in M'$. 
We need to show that $p_{0} \in \overline{\Omega}$ to establish the theorem.  Assume that $p_{0}=0,  q_{0}=0.$
By assumption,  $M'$ is $\ell$-nondegenerate at $q_0$ for some $\ell \geq 1.$ We note that if $1 \leq \ell \leq 2$,   it follows from \autoref{thml2} that $p_{0} \in \overline{\Omega}$. 
We will prove $p_0 \in \overline{\Omega}$ for the general case $\ell \geq 3$  by induction on the order of nondegeneracy $\ell$.

Suppose the statement $p_0 \in \overline{\Omega}$ holds when $\ell\leq k$ for some $k \geq 2.$ We now consider the case $\ell=k+1,$ that is, we assume $M'$ is $(k+1)-$nondegenerate at $q_0.$  Note that if there is a sequence $\{p_{i}\}_{i=1}^{\infty} \subset M$ converging to $p_{0}$ such that $M'$ is at most $k$-nondegenerate at each $F(p_{i})$,  then it follows from the inductive hypothesis that each $p_i \in \overline{\Omega}.$ Consequently, $p_0 \in \overline{\Omega}$ and hence the statement holds. Thus we are only left with the case that there exists a neighborhood
$U$ of $p_{0}$ on $M$ such that $F(U) \subset V_{k}$. Here $V_{k}$ is the set of $k$-degeneracy
of $M'$ near $q_{0}.$ More precisely, 
$$V_{k}=\{q \in V: M'~\text{is}~k\text{-degenerate at}~q \},$$ 
for some small neighborhood  $V=O \cap M'$ of $q_{0}.$ Here $O$ is a small neighborhood of $q_{0}$ in $\mathbb{C}^{n+1}.$ Since $M'$ is $(k+1)$-nondegenerate at $q_0=0$, 
by Lemma \ref{lemma51},  $\Delta_{i_0 i_1 \cdots i_k}(0) \neq 0$, for some
$1 \leq i_{0}\ \leq i_1  \cdots \leq i_{k} \leq n.$  We define 

\begin{equation}\label{eqnphikeqdil}
\varphi_{k}(W, \overline{W})=\det\left(
\begin{array}{c}
\rho_{W} \\
\Lambda_{1}\rho_{W} \\
... \\
\Lambda_{n-1}\rho_{W} \\
\Lambda_{i_1} \cdots \Lambda_{i_k}\rho_{W} \\
\end{array}
\right)(W, \overline{W}).
\end{equation}
Set $$\widetilde{V}_{k}=\{W \in O: \varphi_{k}(W,  \overline{W})=0 \}.$$
We split our argument into two cases.

{\bf Case I:} We first suppose that  $i_{0} \leq n-1.$ Note that $F(U) \subset V_{k} \subset \widetilde{V}_{k}$.
We have

\begin{claim}\label{lmwi0}
The $\overline{w}_{i_{0}}$-derivative $(\varphi_k)_{\overline{w}_{i_0}}$ of $\varphi_{k}$ is nonzero at $q_{0}=0.$
\end{claim}

\begin{proof}
It is equivalent to show that $\Lambda_{i_{0}}\varphi_{k}(0) \neq 0.$ That is, 
\begin{equation}
\Lambda_{i_{0}}\det\left(
                                   \begin{array}{c}
                                     \rho_{W} \\
                                     \Lambda_{1}\rho_{W} \\
                                     ... \\
                                     \Lambda_{n-1}\rho_{W} \\
                                     \Lambda_{i_1} \cdots \Lambda_{i_k}\rho_{W} \\
                                   \end{array}
                                 \right)(0) \neq 0
\end{equation}
Note that
$$\Lambda_{i_{0}}\det\left(
                                   \begin{array}{c}
                                     \rho_{W} \\
                                     \Lambda_{1}\rho_{W} \\
                                     ... \\
                                     \Lambda_{n-1}\rho_{W} \\
                                     \Lambda_{i_1} \cdots \Lambda_{i_k}\rho_{W} \\
                                   \end{array}
                                 \right)(0)=$$
$$
\det\left(
                                   \begin{array}{c}
                                     \Lambda_{i_{0}}\rho_{W} \\
                                     \Lambda_{1}\rho_{W} \\
                                     ... \\
                                     \Lambda_{n-1}\rho_{W} \\
                                     \Lambda_{i_1} \cdots \Lambda_{i_k}\rho_{W} \\
                                   \end{array}
                                 \right)(0)+\det\left(
                                   \begin{array}{c}
                                     \rho_{W} \\
                                     \Lambda_{i_{0}}\Lambda_{1}\rho_{W} \\
                                     ... \\
                                     \Lambda_{n-1}\rho_{W} \\
                                     \Lambda_{i_1} \cdots \Lambda_{i_k}\rho_{W} \\
                                   \end{array}
                                 \right)(0)+ \cdots + \det\left(
                                   \begin{array}{c}
                                     \rho_{W} \\
                                     \Lambda_{1}\rho_{W} \\
                                     ... \\
                                     \Lambda_{i_{0}}\Lambda_{n-1}\rho_{W} \\
                                     \Lambda_{i_1} \cdots \Lambda_{i_k}\rho_{W} \\
                                   \end{array}
                                 \right)(0)+\det\left(
                                   \begin{array}{c}
                                     \rho_{W} \\
                                     \Lambda_{1}\rho_{W} \\
                                     ... \\
                                     \Lambda_{n-1}\rho_{W} \\
                                     \Lambda_{i_{0}}\Lambda_{i_1} \cdots \Lambda_{i_k}\rho_{W} \\
                                   \end{array}
                                 \right)(0)
$$                                                                  

We claim that the first term up to the $n^{\rm th}$ term above are all zero. Indeed,  otherwise,  $M'$ is at most $k-$nondegenerate at $0.$ This is a contradiction to our assumption.

We finally note the last term in the above equation just equals to $\Delta_{i_0i_1 \cdots i_k}(0),$ which is nonzero. This establishes the lemma.
\end{proof}

Recall $F(U) \subset V_{k} \subset \widetilde{V}_{k}.$ We have
\begin{equation}
\varphi_{k}(F(Z), \overline{F(Z)}) \equiv 0, ~\text{for all}~Z \in U \subset M.
\end{equation}
Applying $L_{i_{0}}$ to the above equation and evaluating at $Z=0$, we have, 
\begin{equation}\label{eqnlvarphi2f}
L_{i_{0}}\varphi_{k}(F, \overline{F})|_{0}=\sum_{i=1}^{n+1}(\varphi_{k})_{\overline{w}_{i}}
(F, \overline{F})|_{0}L_{i_{0}}\overline{F}_{i}|_{0}=0.
\end{equation}

Note that by our normalization,   $L_{i_{0}}\overline{F}_{i}(0)=0$,   if $i \neq i_{0}$ and $ L_{i_{0}}\overline{F}_{i_{0}}(0) \neq 0.$ Moreover,  by \autoref{lmwi0},  $(\varphi_k)_{\overline{w}_{i_{0}}}(0) \neq 0.$ This is a contradiction to (\ref{eqnlvarphi2f}). Hence {\bf Case I} cannot happen in this setting.

\bigskip

{\bf Case II:} We are thus only left with the case if $i_0=i_1= \cdots =i_k=n.$  Let $\varphi_k$ be as in (\ref{eqnphikeqdil}) with $i_1= \cdots= i_k=n.$ Again let
$\widetilde{V}_{k}=\{W \in O: \varphi_{k}(W, \overline{W})=0 \}.$

By a similar argument as in the proof of \autoref{lmwi0},  we are able to prove the following lemma.
\begin{claim}\label{lemmawn56}
The $\overline{w}_{n}$-derivative $(\varphi_k)_{\overline{w}_n}$ of $\varphi_{k}$ is nonzero at $0.$
\end{claim}

As a consequence of Lemma \autoref{lemmawn56},  if we define $\widetilde{\varphi}_{k}(W, \overline{W})=\overline{\phi_{k}(W, \overline{W})}$, then 
\begin{equation}\label{eqnphi2wn}
(\widetilde{\varphi}_{k})_{w_{n}}(0)\neq 0.
\end{equation}

Note $F(U) \subset V_{k} \subset \widetilde{V}_{k}.$ We have
$$\varphi_{k}(F(Z), \overline{F(Z)}) \equiv 0, ~\text{for all}~Z \in U.$$
Consequently, 
\begin{equation}\label{eqnphif1}
\widetilde{\varphi}_{k}(F(Z),  \overline{F(Z)}) \equiv 0.
\end{equation}
Recall that for all $Z \in U$,  
\begin{equation}\label{eqnrhof1}
\rho(F, \overline{F})=0, 
\end{equation}

\begin{equation}\label{eqnlrhof1}
L_{i}\rho(F, \overline{F})=0,  1 \leq i \leq n-1.
\end{equation}

Note that
\begin{equation}
\det\left(
  \begin{array}{c}
    \rho_{W}(F,  \overline{F}) \\
    L_{1}\rho_{W}(F,  \overline{F}) \\
    ... \\
    L_{n-1}\rho_{W}(F,  \overline{F}) \\
    (\widetilde{\varphi}_{k})_{W}(F, \overline{F}) \\
  \end{array}
\right) (0) \neq 0
\end{equation}
by equations (\ref{eqnli}), (\ref{eqnrw}), (\ref{eqnphi2wn}).  This implies that equations (\ref{eqnphif1}),  (\ref{eqnrhof1}), (\ref{eqnlrhof1}) forms a nondegenerate system for $F.$ Then by a similar argument as in the proof of \autoref{thml2},  it follows that $F$ is smooth at $0.$ This proves $p_0 \in \overline{\Omega}$ when $\ell=k+1.$ Hence the statement holds for all $l \geq 1$ by mathematical induction. 
\autoref{T1} is thus established. \qed

\section{Proof of Theorem \ref{T1a}}

We are now going to prove \autoref{T1a}.
Fix $p_{0} \in M$ and let $q_{0}=F(p_{0}) \in M'.$ We
will show below that we can apply Theorem 2 for $q_0 \in M' \setminus X$
for some  complex variety $X$ in $\C^{n+1}$. We note that the transversality
of $F$ implies that  the set $F^{-1} (M'\setminus X)$ is open and dense in $M$,  and
the statement of \autoref{T1a} follows. Indeed, to prove that last 
observation, suppose on 
the contrary that for a neighbourhood $U$ 
of $p\in M$ we have that  $F(M\cap U) \subset X$. Then 
$dF (\C T_{p} M)  \subset \C T_{F(p)} X \subset T^{(1,0)}_{F(p)} M' + T^{(0,1)}_{F(p)} M' $,
and so $F$ is not transversal at $p$. 

The following theorem gives the missing claim in the above argument.  Let $V$ be a small neighborhood of $q_{0}$ in $\mathbb{C}^{n+1}.$ We first need to show that

\begin{theorem}\label{thm:lnondegvariety}
$M'$ is finitely  nondegenerate near $q_{0}$  away from a complex analytic variety $X$ in $V.$
\end{theorem}

In order to do so,  we shall first state and prove a useful
general fact. For this,  let $M\subset\CC{N}$ be  a generic
real-analytic submanifold of CR dimension $n$ and real
codimension $d$ (i.e. $N = n+d$). We denote
the set of germs at
$p\in M$ of real-analytic functions on $M$ with $\CC{} \{ M \}_p$.
We say that an
ideal $I\subset \CC{} \{ M \}_p $ is $\bar\partial_b$-closed if
for any CR vector field $L$ on $M$ and any $f\in I$ we
have that $L f \in I$. For any ideal
 $I\subset \CC{} \{ M \}_p $,  we denote by $\mathcal{V} (I)$
  the germ of the
real-analytic subset of $M$ given by the vanishing of all
elements of $I$.

\begin{proposition}
\label{pro:stableideals} Let $I\subset \CC{} \{ M \}_p $ be
a  ideal which is $\bar \partial_b$-closed. Then there
exists a neighborhood $U$ of $p$ in $\CC{N}$ and
a complex subvariety $V\subset U$ such that,  in
the sense of germs at $p$,  $V\cap M = \mathcal{V} (I)$.
\end{proposition}

\begin{proof}
We choose normal coordinates $(z, w) \in \CC{n} \times \CC{d}$
for $M$ at $p$; in these coordinates,  $p=0$ and $M$
is defined
by
\[  w = Q (z, \bar z,  \bar w),  \]
where $Q = (Q^1, \dots, Q^d)$ is a holomorphic map
with values in $\CC{d}$,  defined in a neighborhood
of $(0, 0, 0) \in \CC{n} \times \CC{n} \times \CC{d}$, 
satisfying
\begin{equation}
\label{e:normality} Q(z, 0, \bar w) = Q (0,  \bar z ,  \bar w ) = \bar w,  \quad Q(z, \bar z, \bar Q (\bar z,  z,  w)) = w.
\end{equation}
A basis of the CR vector fields on $M$ near $0$ is given
by
\[  L_j = \frac{\partial}{\partial \bar z_j} + \sum_{k=1}^d
\bar Q^k_{\bar z_j} (\bar z,  z,  w) \frac{\partial}{\partial \bar w_k}.  \]
As usual,  we use multtiindex
 notation and for $\alpha = (\alpha_1,  \dots,  \alpha_n)$
 we  write
  $ L^\alpha =  L_1^{\alpha_1} \cdots  L_n^{\alpha_n}$.

Let $f\in \CC{} \{M\}_p  $. There
exists a holomorphic function $F(z, w, \chi, \tau)$ defined
in a neighborhood of $(0, 0, 0, 0) \in \CC{n} \times \CC{d} \times \CC{n} \times \CC{d}$ such that
$f(z, w, \bar z,  \bar w) = F (z,  w,  \bar z,  \bar Q (\bar z,  z,  w))$ for
$(z, w) \in M$. For any such $f$,  we write $\varphi_f (z, w, \chi) = F (z,  w,  \bar z,  \bar Q (\bar z,  z,  w))$.  We note
that $\frac{\partial^{|\alpha|} \varphi_f}{\partial \chi^\alpha} (z, w, \bar z) = \varphi_{L^\alpha f} (z, w, \bar z)$.
We also note that we
can write
\begin{equation}
\label{e:taylorexp}
\begin{aligned}
F(z, w, \chi,  \bar Q (\chi,  z,  w) ) &= \sum_{\alpha} \frac{1}{\alpha!}
\frac{\partial^{|\alpha|}}{\partial \chi^\alpha}
F(z, w, \chi, \bar Q(\chi, z, w))  \biggr|_{\chi = 0}  \chi^\alpha \\
&=
\sum_\alpha  \frac{1}{\alpha!}
 L^\alpha F (z, w, 0, w)  \chi^\alpha \\
&= \sum_{\alpha} \frac{1}{\alpha!}
\frac{\partial^{|\alpha|} \varphi_f}{\partial \chi^\alpha} (z, w, 0)  \chi^\alpha .
\end{aligned}
\end{equation}

So assume that we have chosen a small neighborhood of $0$,  such
that inside this neighborhood,  $\mathcal{V}(I)$
is defined by an ideal $\tilde I$ of functions $f(z, w, \bar z , \bar w)$ extending holomorphically to a common neighborhood of
 $(0, 0, 0, 0) \in \CC{n} \times \CC{d} \times \CC{n} \times \CC{d}$.
 We claim that $\mathcal{V} (\tilde I)  = \{(z, w) \colon \varphi_f (z, w, 0) = 0,  \,  f\in\tilde I\} $.

 Let $Z_0 = (z_0,  w_0) \in \mathcal{V} (\tilde I)$,  and
 let $f \in \tilde I$. Then the holomorphic function
 $\chi \mapsto \varphi_f (z_0, w_0, \chi)$ vanishes to
 infinite order at $\chi = \bar z_0$; hence also
 $\varphi_f(z_0, w_0, 0) = 0$. Assume now that
 $\varphi_g (z_0, w_0, 0) = 0$ for every  $g\in\tilde I$. Then
 by \eqref{e:taylorexp},  if $f\in\tilde I$ is arbitrary, 
 then $f(z_0,  w_0 ,  \bar z_0,  \bar w_0 ) = 0$. Hence,  $\mathcal{V} (\tilde I)  = \{(z, w) \colon \varphi_f (z, w, 0) = 0,  \,  f\in\tilde I\} $ as
 claimed.
\end{proof}

The proof of Theorem~\ref{thm:lnondegvariety} is a combination
of Proposition~\ref{pro:stableideals} with the following fact.

\begin{lemma}
\label{lem:nondegvariety2} Let $X \subset M$ be the set of
points $p$ in $M$ at which $M$ is not finitely nondegenerate
of any order $k$. Then $X$ can be defined,  near every
point $p\in M$,  by an ideal which is $\bar \partial_b$-closed.
\end{lemma}

\begin{proof}
Let $p\in M$,  and let $Z = (Z_1, \dots , Z_N)$ be coordinates near
$p$. We note that $M$ is $k$-nondegenerate if the
space $E_k (p)$ has dimension $N$,  where
\[ E_0 = \Gamma(M,  T^0 M),  \quad
E_k = E_{k-1} + \{\mathcal{L}_{L} \omega \colon \omega \in E_{k-1},   L \text{ CR} \}.\]
Here $T^0 M = {\rm Re} (T^{(1,0)} M)^\perp \cap (T^{(0,1)} M )^\perp$ denotes the (real) characteristic bundle of $M$ and
$\mathcal{L}$ the Lie derivative (of forms). It turns out
that $E_k \subset \Gamma(M, T'M)$,  where $T'M =(T^{(0,1)} M)^\perp $ is the bundle
of holomorphic forms on $M$. We have that
\[ T'M = \langle dZ_1,  \dots , dZ_N \rangle. \]
We note that for
\[ \omega = \sum_{j=1}^N \omega^j dZ_j,  \]
it holds that
\[ \mathcal{L}_{L} \omega = \sum_{j=1}^N (L\omega^j) dZ_j .\]
Choose a basis of characteristic forms
 $\theta_j= \sum_{k=1}^N  \theta_j^k d Z_k $,  where $j = 1, \dots,  d$.
 The space $E_k$ is therefore
 spanned by forms of the form
 \[ \mathcal{L}_{L^\alpha} \theta_j =
 \mathcal{L}_{L_1}^{\alpha_1} \cdots
 \mathcal{L}_{L_n}^{\alpha_n} \theta_j
 = \sum_{k=1}^N  ({L}^\alpha \theta_j^k )d Z_k , \quad j=1, \dots,  d ,  \quad |\alpha|\leq k. \]

 We therefore have that $M$ is not $\ell$-nondegenerate for
 some $\ell\leq k_0$ at $p$
 if and only if for every choice $r = (r_1 , \dots,  r_N)$ of
 integers $r_k \in \{1, \dots , d \} $ and
 for every choice of
  multiindeces $A = (\alpha^1,  \dots , \alpha^N)   $, 
  where $ \alpha^j = (\alpha^j_1,  \dots ,  \alpha^j_N) $ satisfies
  $|A| = \max\{ |\alpha^j| \colon j=1, \dots, N \} \leq k_0$, 
  the determinant
  \[ D(A, r) = \det
  \begin{pmatrix}
   	{L}^{\alpha^1} \theta_{r_1}^1 &
   	\dots & {L}^{\alpha^1} \theta_{r_1}^N \\
   	\vdots & & \vdots \\
   	{L}^{\alpha^N} \theta_{r_N}^1 &
   	\dots & {L}^{\alpha^N} \theta_{r_N}^N
   \end{pmatrix} \]
vanishes at $p$; that is,  if we denote by $X_{k_0}$
the set of all points where $p$ is not $\ell$-nondegenerate
where $\ell \leq k_0$,  then $X_{k_0}$ is defined
by the ideal
\[ I_{k_0} = \left(\{ D(A, r) \colon |A| \leq k_0 \}\right). \]
Note that $L I_k \subset I_{k+1}$.
The set $X = \cap_k X_k$ is now defined by
$I = \cup_k I_k $,  which is $\bar \partial_b$-closed.
\end{proof}

By combining \autoref{pro:stableideals} and \autoref{lem:nondegvariety2},  we obtain the result in \autoref{thm:lnondegvariety}.
Now the proof of \autoref{T1a} follows by combining \autoref{thm:lnondegvariety} and the argument in the beginning of the section.

\section{Proof of Theorem~\ref{T2}} 
\label{sec:proof_of_theorem_}

\autoref{T2} can be obtained very much in the same way as 
\autoref{T1a} from \autoref{T1}, noting that the statement of 
\autoref{T1a} is also valid in the real-analytic category. However, we 
are grateful to one of the anonymous referees, to point out another, 
conceptually very nice proof
using the methods of the paper by Mir \cite{M1}, and using 
our normalization given in \autoref{rknormalization}. We  sketch  
that proof here, referring the reader for the 
details of the steps outlined in the cases below to \cite{M1}. 

From \autoref{rknormalization}, we have 
\begin{equation}\label{eqn1}
\rho(F(Z), \overline{F(Z)})=-\mathrm{Im}F_{n+1}(Z) +\sum_{j=1}^{n-1}|F_{j}(Z)|^2 + \phi(F(Z), \overline{F(Z)}),  \quad \phi=O(|F(Z)|^3). 
\end{equation} 
Applying a basis of CR vector fields $L_1 , \dots L_{n-1}$ of $M$ 
to that equation, and using the 
implicit function theorem in the resulting system of
 $n$ equations, we 
obtain a holomorphic map $\Psi$, valued in $\C^n$,  defined in a neighbourhood
of $(0, L \bar F (0), 0) \in \C^{n+1} \times \C^{ (n+1)\times (n-1)} \times \C$,
such that with $\tilde F = (F_1, \dots ,F_{n-1}, F_{n+1})$ the following
holds on $M$:
\begin{equation}
 	\label{e:eqn2} \tilde F = \Psi ( \bar F , L \bar F, F_n) = \left( \Psi_1 ( \bar F , L \bar F, F_n), \dots ,\Psi_n ( \bar F , L \bar F, F_n) \right) .  
 \end{equation} 

We now distinguish the following two cases: 

{\bf Case 1.} There exists a CR vector field $X$ on $M$ and a ${j_0}$ with 
$1 \leq {j_0} \leq n$ such 
that 
$X \Psi_{j_0} (\bar F , L \bar F, T) \not\equiv 0$ in a neighbourhood of $(0,0) \in M \times \C_T$. In 
that case, one can apply the implicit function theorem to the system of  equations formed by the  equation $0 = X \Psi_{j_0} (\bar F , L \bar F, F_n)$
and the equations \eqref{e:eqn2}, at a generic point in $M$, to see that
$F$ is real-analytic there; hence we see that in Case 1, $F$ is 
real-analytic on a dense, open subset of $M$. 

{\bf Case 2.} For all CR vector fields $X$ on $M$ and all $j$, $\leq j \leq n$, we have $X \Psi_j (\bar F , L \bar F, T) \equiv 0$. In that case, the reflection principle implies that $\Psi (\bar F (Z) , L \bar F (Z), T) =: \Phi (Z, T)$ is holomorphic in a neighbourhood 
of $0\in\C^n \times \C$.  We write $\tilde \Phi (Z, T) = (\Phi_{1} (Z,T) ,\dots , \Phi_{N-1} (Z, T), T, \Phi_N(Z,T))$ and distinguish the following two alternatives: 

{\bf Case 2A.} $\varrho (\tilde \Phi (Z, T) , \overline{\tilde \Phi (Z,T)}) \not\equiv 0$ for $Z\in M$ and $T\in \C$: In that case, as in Case 1, it follows 
that actually $F$ is real-analytic in a dense open subset of $M$.

{\bf Case 2B.} $\varrho (\tilde \Phi (Z, T) , \overline{\tilde \Phi (Z,T)}) \equiv 0$ for $Z\in M$ and $T\in \C$: By our normalization of $F$, we 
see that $\tilde \Phi $ is a {\em biholomorphism} from $M\times \C$ into $M'$. 
This is impossible if we assume $M'$ to be holomorphically nondegenerate. 
Hence, Case 2B does not happen, and in any case, $F$ is real-analytic on a dense 
open subset of $M$.





\end{document}